\documentclass{amsart}  %used for final

\usepackage[latin1]{inputenc}
\usepackage{amsmath} 
\usepackage{amsfonts}
\usepackage{amssymb}
\usepackage{stmaryrd}
\usepackage{latexsym} 
\usepackage{graphicx}
\usepackage{subfigure}
\usepackage{color}
\usepackage{verbatim}
\usepackage[all]{xy}
\usepackage{graphics}
\usepackage{pdfsync}
\usepackage[backref=page]{hyperref}
\usepackage{ytableau}

\theoremstyle{plain}
\newtheorem{theorem}{Theorem}[section]

\newtheorem{lemma}[theorem]{Lemma}
\newtheorem{corollary}[theorem]{Corollary}

\newtheorem{example}[theorem]{Example}

\theoremstyle{remark}
\newtheorem{remark}[theorem]{Remark}

\numberwithin{equation}{section}
\newcommand{\suchthat}{\;|\;}

\newlength\cellsize 
\savebox2{%
\begin{picture}(15,15)
\put(0,0){\line(1,0){15}}
\put(0,0){\line(0,1){15}}
\put(15,0){\line(0,1){15}}
\put(0,15){\line(1,0){15}}
\end{picture}}
\newcommand\cellify[1]{\def\thearg{#1}\def\nothing{}%
\ifx\thearg\nothing
\vrule width0pt height\cellsize depth0pt\else
\hbox to 0pt{\usebox2\hss}\fi%
\vbox to 15\unitlength{
\vss
\hbox to 15\unitlength{\hss$#1$\hss}
\vss}}
\newcommand\tableau[1]{\vtop{\let\\=\cr
\setlength\baselineskip{-16000pt}
\setlength\lineskiplimit{16000pt}
\setlength\lineskip{0pt}
\halign{&\cellify{##}\cr#1\crcr}}}
\savebox3{%
\begin{picture}(15,15)
\put(0,0){\line(1,0){15}}
\put(0,0){\line(0,1){15}}
\put(15,0){\line(0,1){15}}
\put(0,15){\line(1,0){15}}
\end{picture}}
\newcommand\expath[1]{%
\hbox to 0pt{\usebox3\hss}%
\vbox to 15\unitlength{
\vss
\hbox to 15\unitlength{\hss$#1$\hss}
\vss}}
\newcommand\bas[1]{\omit \vbox to \cellsize{ \vss \hbox to \cellsize{\hss$#1$\hss} \vss}}

 \usepackage{chessfss}
\usepackage{epstopdf}

\usepackage{tikz}
\usetikzlibrary{arrows,petri,topaths}
\usepackage{tkz-berge}
\usepackage{xcolor}
\usepackage{tikz-qtree}
\usetikzlibrary{positioning,arrows,calc,intersections,shapes,decorations}

\newcommand{\Asc}{\mathrm{Asc}} % ascent set
\newcommand{\Dsc}{\mathrm{Des}} % descent set
\newcommand{\Exc}{\mathrm{Exc}}  % strong excedance set
\newcommand{\nExc}{\mathrm{nExc}} % not a strong excedance set

\newcommand{\asc}{\mathrm{asc}} % number of ascents
\newcommand{\dsc}{\mathrm{des}} % number of descents
\newcommand{\exc}{\mathrm{exc}}  % number of strong excedance
\newcommand{\nexc}{\mathrm{nexc}} % number of not a strong excedance
\newcommand{\sub}{\mathrm{sub}} % subcedances

\newcommand{\Func}{\mathcal{F}}
\newcommand{\wtFunc}{\widetilde{\Func}}

\newcommand{\FFunc}{\mathcal{FF}}

\newcommand{\tree}{\mathcal{T}} % rooted labeled trees
\newcommand{\ptree}{\mathcal{PT}} % rooted labeled plane k-ary trees
\newcommand{\iptree}{\mathcal{IPT}} % increasing k-ary trees
\newcommand{\dtree}{\mathcal{DT}}  %decorated rooted labeled trees
\newcommand{\ltree}{\mathcal{LST}} % local search trees
\newcommand{\riptree}{\mathcal{RPT}} %right increasing plane tree

\newcommand{\spine}{\mathrm{spine}} % spine of tree
\newcommand{\rec}{\mathrm{rec}}% records of a tree

\newcommand{\hypcat}{\mathcal{C}} % Catalan
\newcommand{\hyplin}{\mathcal{L}} % Linial
\newcommand{\hypshi}{\mathcal{S}} % Shi

\newcommand{\alpu}{\mathsf{u}}
\newcommand{\alpv}{\mathsf{v}}

\begin{document}

\title[Gessel polynomials and extended Linial arrangements]{Gessel polynomials, rooks, and extended Linial arrangements}

\author{Vasu Tewari}
\address{Department of Mathematics, University of Washington, Seattle, WA 98195, USA}
\email{\href{mailto:vasut@math.washington.edu}{vasut@math.washington.edu}}

%\subjclass[2010]{Primary 05E05, 20C08; Secondary 05A05, 05E10, 05E15, 06A07, 16T05, 20C30}

\begin{abstract}
We study a family of polynomials associated with ascent-descent statistics on labeled rooted plane $k$-ary trees introduced by Gessel, from a rook-theoretic perspective. We generalize the excedance statistic on permutations to maximal nonattacking rook placements on certain rectangular boards by decomposing them into boards of staircase shape. We then relate the number of maximal nonattacking rook placements on certain skew boards to the number of regions in extended Linial arrangements by establishing a relation between the factorial polynomial of those boards to the characteristic polynomial of extended Linial arrangements. Furthermore, we give a combinatorial interpretation of the number of bounded regions in extended Linial arrangements in the setting of labeled rooted plane $k$-ary trees. Finally, using the work of Goldman-Joichi-White, we identify graphs whose chromatic polynomials equal the characteristic polynomials of extended Linial arrangements upto a straightforward normalization.
\end{abstract}

\maketitle
%\tableofcontents

% Introduction
\section{Introduction}
This article concerns a family of multivariate polynomials associated with ascents and descents on labeled rooted plane $k$-ary trees that was introduced by Gessel \cite{Gessel-unpublished} and studied by \cite{Drake-thesis,Kalikow}. Gessel further observed that various evaluations of these polynomials counted the number of regions in well-known deformations of Coxeter hyperplane arrangements, and asked for a bijective/combinatorial explanation \cite{Gessel-Oberwolfach}. In what follows, we describe progress in this direction. %Other work in the same direction has been done in \cite{Corteel-Forge-Ventos,Forge}.

We study this family of polynomials from a rook-theoretic perspective, drawing inspiration from recent work of Lewis-Morales \cite{Lewis-Morales} studying diagrams of permutations as well as the work of Leven-Rhoades-Wilson \cite{Leven-Rhoades-Wilson} relating rook placements and Shi-Ish arrangements.
%\cite{Sjostrand} relating rook placements on skew boards to Poincar\'e polynomials of certain intervals in the Bruhat order. %More specifically, we associate certain composition-valued statistics with nonattacking maximal rook placements of boards of size $(n-1)\times kn$ where $n$ is a positive integer, and study their joint distribution. These statistics are a natural generalization of the classical excedance and subcedance statistics associated with permutations. In fact, the case of nonattacking maximal rook placements on an $(n-1)\times n$ board coincides with the case of permutations.
%We then proceed to mention the consequences of our bijection to the study of certain deformations of the Coxeter arrangements; more specifically, the extended Linial arrangements defined by Postnikov-Stanley \cite{Postnikov-Stanley}.

To give a flavor of our main results, we establish some notation. Let $n\geq 1$. A local binary search tree on $n$ nodes is defined to be a labeled rooted plane binary tree on $n$ nodes with labels drawn from $[n]=\{1,\ldots, n\}$ such that the label of the left child (right child) is always smaller (respectively greater) than the parent. The Linial arrangement $\hyplin_{n-1}$ is the real hyperplane arrangement comprising the hyperplanes $x_i-x_j=1$ for $1\leq i<j\leq n$ considered in the ambient vector space $V_{n-1}$ consisting of points $(x_1,\ldots,x_n)\in \mathbb{R}^n$ such that $x_1+\cdots +x_n=0$.  A special case of one of our main theorems is the following result.
\begin{theorem}\label{thm: intro 1}
The number of bounded regions in the Linial arrangement $\hyplin_{n-1}$ equals the following quantities.
\begin{enumerate}
\item The number of local binary search trees on $n$ nodes such that the nodes labeled $1$ and $n$ are both leaves. Here the root node is not considered to be a leaf node.
\item The number of nonattacking rook placements of $n-1$ rooks on the skew Ferrers board of shape given by  $(2n-3,\ldots,n-1)/(n-1,\ldots,1)$.
\item The number of sequences $(x_1,\ldots, x_{n-1})$ of positive integers satisfying $1\leq x_i\leq n-2$ for $1\leq i\leq n-1$, and $x_i+i\neq x_j+j$ for $1\leq i\neq j\leq n-1$.
\end{enumerate}
\end{theorem}
A few remarks about the above theorem are in order. The question of finding a combinatorial meaning to the number of bounded regions of the Linial arrangement in the setting of labeled trees has been raised in \cite{Athanasiadis-Advances} (cf. discussion after Theorem 4.2 therein). To the best of the author's knowledge, there have not been any results in this direction.
Secondly, that local binary search trees have to do with regions in the Linial arrangement is no surprise. Indeed, Postnikov \cite{Postnikov-JCTA} proved that the number of regions in the Linial arrangement equals the number of local binary search trees (in fact, Postnikov worked with intransitive trees or alternating trees, but establishes a bijection between them and local binary search trees \cite[Section 4.1]{Postnikov-JCTA}). As the theorem above shows, the number of bounded regions in the Linial arrangement is counted by a fairly nicely described subset of the set of local binary search trees. Starting from $n=1$, the first few terms of this sequence (not present in the OEIS) are: $0$, $0$, $1$, $4$, $26$, $212$, $2108$, $24720$.

Our second main result relates the factorial polynomial of a specific skew Ferrers board to the characteristic polynomial of the Linial arrangement $\hyplin_{n-1}$.
\begin{theorem}\label{thm: intro 2}
Let $B$ denote the skew Ferrers board of shape $(2n-3,\ldots,n-1)/(n-1,\ldots,1)$.
Let $R(t,B)$ denote the factorial polynomial of $B$, and let $\chi_{\hyplin_{n-1}}(t)$ denote the characteristic polynomial of the arrangement $\hyplin_{n-1}$. Then we have the following equality.
\begin{align*}
R(t,B)=(-1)^{n-1}\chi_{\hyplin_{n-1}}(1-t)
\end{align*}
\end{theorem}
The outline of this article is as follows: We begin by defining notions related to hyperplane arrangements in Section \ref{sec: hyp arrangements}. We follow it up by defining rook-theoretic notions in Section \ref{sec:rook placements}. We also introduce certain special rectangular boards and define generalized excedance and subcedance statistics associated with maximal nonattacking rook placements on these boards. In Section \ref{sec:trees}, we introduce labeled rooted plane $k$-ary trees amongst other notions. We then define ascent-descent statistic on these trees and proceed to give our bijection $\Psi$ that takes maximal nonattacking rook placements on $(n-1)\times kn$ boards to labeled rooted plane $k$-ary trees on $[n]$. This bijection further has the property that it takes the $(\exc, \sub)$ associated with the rook placement to the $(\asc,\dsc)$ pair associated with the corresponding tree. This given, we define the Gessel polynomials. We use our bijection  to connect rook placements on certain skew Ferrers boards with the number of regions in extended Linial arrangements. In Section \ref{sec: chromatic poly} we identify graphs whose chromatic polynomials are related to characteristic polynomials of extended Linial arrangements. Finally, we conclude with a  brief discussion on future avenues in Section \ref{sec:final}.

\section{Hyperplane arrangements}\label{sec: hyp arrangements}
We begin by defining some basic notions concerning hyperplane arrangements. Our exposition will be brief; the reader is referred to \cite{Orlik-Terao,Stanley-notes} for more details. We will be faithful to the terminology used in \cite{Postnikov-Stanley} to a large extent.

A hyperplane arrangement is a finite collection of affine hyperplanes in a vector space. Given a finite-dimensional vector space $V$ over $\mathbb{R}$, let $\mathcal{A}$ be a hyperplane arrangement therein. We will denote the number of regions of $\mathcal{A}$, that is, the number of connected components of $V-\cup_{H\in \mathcal{A}}H$, by $r(\mathcal{A})$. Furthermore, the number of relatively bounded regions will be denoted by $b(\mathcal{A})$.

The poset that connects the study of hyperplane arrangements to combinatorics is the \emph{intersection poset } $L_{\mathcal{A}}$ associated with a hyperplane arrangement $\mathcal{A}$. 
It consists of all nonempty intersections $H_{i_1}\cap \cdots \cap H_{i_k}$ of hyperplanes in $A$, ordered by reverse inclusion. Clearly, the unique  minimum element $\hat{0}$ in $L_{\mathcal{A}}$ is the ambient vector space $V$. Given $L_{\mathcal{A}}$, we can define the \emph{characteristic polynomial} of $\mathcal{A}$, denoted by $\chi_{\mathcal{A}}(q)$, as shown below.
\begin{align*}
\chi_{\mathcal{A}}(q)=\sum_{z\in L_{\mathcal{A}}}\mu(\hat{0},z)q^{\dim(z)}
\end{align*}
Here $\mu$ denotes the M\"obius function of the poset $L_\mathcal{A}$.

Another polynomial encoding topological information about the hyperplane arrangement $\mathcal{A}$ is the \emph{Poincar\'e polynomial} $\mathrm{Poin}_{A}(q)$ defined below.
\begin{align*}
\mathrm{Poin}_{\mathcal{A}}(q)=\sum_{m\geq 0}\dim H^{k}(C_{\mathcal{A}};\mathbb{C})q^k
\end{align*}
Here $C_{\mathcal{A}}$ denotes the complement of the complexified arrangement $\mathcal{A}_{\mathbb{C}}$ in the complexified vector space $V_{\mathbb{C}}$. Furthermore, $H^{k}(\text{ };\mathbb{C})$ denotes the singular cohomology with coefficients belonging to $\mathbb{C}$. The Poincar\'e polynomial $\mathrm{Poin}_{\mathcal{A}}(q)$ is a $q$-analogue of $r(\mathcal{A})$ as is made precise by the following theorem of Orlik-Solomon \cite{Orlik-Solomon-Inventiones}.
\begin{theorem}
We have that $r(\mathcal{A})=\mathrm{Poin}_{\mathcal{A}}(1)$ and $b(\mathcal{A})=\mathrm{Poin}_{\mathcal{A}}(-1)$.
\end{theorem}
Compare the theorem above with the following seminal result proved by Zaslavsky \cite{Zaslavsky}.
\begin{theorem}\label{thm: Zaslav}
Consider a hyperplane arrangement $\mathcal{A}$ in $\mathbb{R}^n$.  Then we have that $r(\mathcal{A})=(-1)^n\chi_{\mathcal{A}}(-1)$ and $b(\mathcal{A})=(-1)^n\chi_{\mathcal{A}}(1)$.
\end{theorem}

The similarity between the two theorems above was explained by Orlik-Solomon \cite{Orlik-Solomon-Inventiones} who proved that
\begin{align*}
\chi_{\mathcal{A}}(q)=q^d\mathrm{Poin}_{\mathcal{A}}(-q^{-1}),
\end{align*}
where $d=\dim (V)$.

We now turn our focus to special classes of hyperplane arrangements that are our primary object of study. Given a positive integer $n$, let $V_{n-1}$ be the subspace in $\mathbb{R}^n$ comprising vectors $(x_1,\ldots,x_n)$ such that $\sum_{i=1}^nx_i=0$. The hyperplane arrangements we consider in this paper will all lie in $V_{n-1}$ and, in fact, will be special instances of what are called \emph{truncated affine arrangements} in \cite{Postnikov-Stanley}.
This given, since the hyperplane arrangements that we will consider are \emph{essential}, we will not be using the term relatively bounded for our results; instead we will just use the term bounded (see Page 3 in \cite{Stanley-notes}).

Fix nonegative integers $a$ and $b$  such that $a+b\geq 2$. Consider the \emph{truncated affine arrangement} $\mathcal{A}_{n-1}^{a,b}$ given by
\begin{align*}
x_i-x_j=-a+1,-a+2,\ldots, b-1, \hspace{5mm} 1\leq i< j\leq n.
\end{align*}
We will denote the characteristic polynomial of $\mathcal{A}_{n-1}^{a,b}$ by $\chi_{n-1}^{a,b}(q)$.

Instances of  truncated affine arrangements have been a focus of intense study in recent years, and have resulted in a lot of interesting combinatorics. We describe these arrangements next.
Fix a positive integer $a\geq 1$. The arrangement $\mathcal{A}^{a,a}_{n-1}$  will be called the \emph{extended Catalan arrangement}, and will be denoted by $\hypcat_{n-1,a}$. Note that the case $a=1$ corresponds to the braid arrangement, while the case $a=2$ corresponds to the classical Catalan arrangement. The arrangement $\mathcal{A}^{a,a+1}_{n-1}$ will be referred to as the \emph{extended Shi arrangement}, and will be denoted by $\hypshi_{n-1,a}$. The case $a=1$ corresponds the extremely well-studied Shi arrangement \cite{Athanasiadis-Linusson, Headley, Postnikov-Stanley,Shi-Lecture notes, Shi-JLMS,Stanley-Pnas}.
Finally the arrangement $\mathcal{A}^{a-1,a+1}_{n-1}$ will be referred to as the \emph{extended Linial arrangement}, and will be denoted by $\hyplin_{n-1,a}$. The case $a=1$ corresponds to the classical Linial arrangement mentioned in the introduction of this article. From now onwards $\hyplin_{n-1}:=\hyplin_{n-1,1}$. Finally, before proceeding to our next section, we would like to mention that various aspects of the truncated affine arrangements have been studied in great depth in \cite{Athanasiadis-survey, Athanasiadis-JACo,Athanasiadis-EuJC,Postnikov-Stanley} and the reader is referred to them for further beautiful combinatorics surrounding them.

\iffalse
A question posed in \cite{Athanasiadis-Advances} is the following: Give a combinatorial interpretation for the number of bounded regions in the Linial arrangement.
In what follows we will give multiple interpretations to the number of bounded regions of the Linial arrangement.
But, unfortunately, there will be no explicit bijection.

From Zaslavsky's formula applied to the characteristic polynomial of the Linial arrangement, we have the following.
\fi

\section{Rook placements}\label{sec:rook placements}
A \emph{board} $B$ is a finite subset of $\mathbb{P}\times \mathbb{P}$ (thought of in Cartesian coordinates) where $\mathbb{P}$ denotes the set of positive integers. A \emph{nonattacking rook placement} on $B$ is one where there is at most one rook in any row or column. A nonattacking rook placement is called \emph{maximal} if there is exactly one rook in every row, and at most one in any column. We will label the rows and columns of a board with positive integers. By row $i$, we mean the $i$-th row from the bottom, while by column $j$ we mean the $j$-th column from the left. For the purposes of this article, we will primarily be interested in \emph{skew Ferrers boards}. These are boards corresponding to Young diagrams of shape $\lambda/\mu$ for partitions $\lambda$ and $\mu$ satisfying $\mu\subseteq \lambda$. 

%A \emph{file placement} of rooks on $B$ is a rook placement such that there is at most one rook in every row. A \emph{maximal file placement} is a rook such that there is exactly one rook in every row.

Given a board $B$ and a nonnegative integer $k$, let $r_k(B)$ denote the number of nonattacking rook placements of $k$ rooks on $B$. Clearly if $k$ exceeds the number of rows in $B$ then $r_{k}(B)=0$. Also, we have that $r_0(B)=1$. Finally, if $m\geq $ number of rows in $B$, we denote by $R_m(x,B)$ the \emph{$m$-factorial polynomial} of $B$ \cite{GJW-iii} defined below.
\begin{align*}
R_m(x,B)=\sum_{k\geq 0}r_k(B)(x)_{m-k}
\end{align*}
Here $(x)_{j}=x(x-1)\cdots (x-j+1)$ is the usual \emph{falling factorial}. We will normally set $m$ equal to the number of rows in the board $B$. In this case we will refer to the $m$-factorial polynomial of $B$ as the \emph{factorial polynomial} of $B$, and denote it by $R(x,B)$.

\begin{example}\label{example: board b}
\em{Consider the board $B$ below which has $3$ rows.
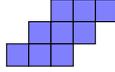
\begin{align*}
\ytableausetup{mathmode, boxsize=0.8em}
\begin{ytableau}
\none & \none & \none & *(blue!50) & *(blue!50)& *(blue!50)\\
\none & \none & *(blue!50) & *(blue!50)& *(blue!50)\\
\none & *(blue!50) & *(blue!50)& *(blue!50)\\
\end{ytableau}
\end{align*}
 Note that from its definition it follows that $R(x,B)=(x)_3+9(x)_2+22(x)_1+14$. It is worth mentioning that the coefficients $1,9,22,$ and $14$ appear in the analysis of the `three-ply' staircase in the work of Riordan \cite{Riordan}. Of course, the above board is a special instance of the boards considered by Riordan.}
\end{example}
The theorem below, which is present in \cite{GJW-v}, will be useful for us. Let $\Pi_m$ denote the lattice of partitions of $[m]$, and let $\hat{0}$ denote the unique minimum in $\Pi_m$. Consider a board $B$ with $m$ rows. Given a subset $S \subseteq [m]$, we define a statistic $v_{B}(S)$ as follows.
\begin{align*}
v_B(S)=|\{j\suchthat (i,j)\in B \hspace{2mm}\forall i\in S\}|
\end{align*}
\begin{theorem}[Goldman-Joichi-White]\label{thm: GJW-v}
Let $B$ be a board with $m$ rows. Then
\begin{align*}
R(x,B)=\sum_{\sigma\in \Pi_m}\mu(\hat{0},\sigma)\prod_{A\in \sigma}(x+v_{B}(A))
\end{align*}
where $\mu$ denotes the M\"obius function of $\Pi_{m}$. Also, by $A\in\sigma$ we mean that $A$ is a block in $\sigma$.
\end{theorem}
\begin{example}\em{
Given below is the Hasse diagram of the partition lattice $\Pi_3$.
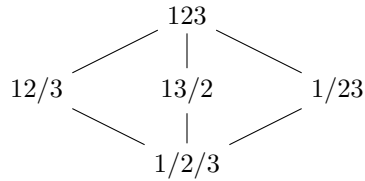
\begin{align*}
\begin{tikzpicture}
        \tikzstyle{every node} = [rectangle]
        \node (s) at (0,-3) {$1/2/3$};
        \node (s1) at (-2,-2) {$12/3$};
        \node (s3) at (0,-2) {$13/2$};
        \node (s2) at (2,-2) {$1/23$};
        \node (s12) at (0,-1) {$123$};
        \foreach \from/\to in {s/s1, s/s2, s1/s12, s2/s12,s/s3,s3/s12}
            \draw[-] (\from) -- (\to);
\end{tikzpicture}
\end{align*}
In writing the blocks of a set partition above, we have omitted commas and braces.

Then, by Theorem \ref{thm: GJW-v}, for the same board $B$ as in Example \ref{example: board b}, we obtain that $R(x,B)=(x+3)^3-(x+2)(x+3)-(x+1)(x+3)-(x+2)(x+3)+2(x+1)=x^3 + 6x^2 + 15x + 14$. Note that $R(1,B)=36$ and $R(-1,B)=4$. To hint at where we are going with this example, note further that 36 and 4 are the number of regions and the number of bounded regions in the Linial arrangement $\hyplin_3$ respectively. }
\end{example}
We will also be considering rook placements on rectangular boards in terms of functions as described next.
Given positive integers $n,k$, let $\Func _{n,k}$ denote the following set.
\begin{align*}
\Func_{n,k}=\{f:[n-1] \hookrightarrow [kn] \}
\end{align*}
That is, the elements of $\Func_{n,k}$ are injective functions from the first $n-1$ positive integers to the first $kn$ positive integers. 
We will identify an element $f\in\Func_{n,k}$ with the rook placement on the rectangular board with $n-1$ rows and $kn$ columns obtained by placing a rook in the position $(i,f(i))$ for $1\leq i\leq n-1$. 
Note that this is a maximal nonattacking rook placement. 
Furthermore, statements involving rook placements on subboards of the $(n-1)\times kn$ board can be phrased in terms of elements of $\Func_{n,k}$ satisfying certain constraints. Finally, we note here that there is a unique injective function from the empty set $\emptyset$ (corresponding to the case $n=1$)  into $[k]$ for $k$ a positive integer. This corresponds the unique rook placement of $0$ rooks on the empty board. To avoid unnecessary complications with the case $n=1$, we will assume throughout that $n\geq 2$.

Closely associated to the set $\Func_{n,k}$ is the set $\wtFunc_{n,k}$ defined below.
\begin{align*}
\wtFunc_{n,k}=\{f: [n-1] \hookrightarrow [n] \times [k]\}
\end{align*}
The sets $\wtFunc_{n,k}$ and $\Func_{n,k}$ are clearly equinumerous, and we will use the bijection $\Phi$ described next to associate an element of $\wtFunc_{n,k}$ with $\Func_{n,k}$. The image of $f\in \Func_{n,k}$ under $\Phi$ is the element $\Phi(f)\in \wtFunc_{n,k}$ defined as follows.
\begin{align*}
\Phi(f) (i) = (f(i)-n\left\lfloor \frac{f(i)-1}{n} \right\rfloor, \left\lceil \frac{f(i)}{n} \right\rceil) 
\end{align*}
It is easy to see that $\Phi$ establishes a bijection between $\Func_{n,k}$ and $\wtFunc_{n,k}$.

For any  element $g\in \wtFunc_{n,k}$ we define $1\leq i\leq n-1$ to be a \emph{$j$-excedance} (respectively \emph{$j$-subcedance}) if $g(i)$ equals $(a,j)$ for some positive integer $a$ satisfying $a>i$ (respectively $a\leq i$). This given we associate two weak compositions with $g\in \wtFunc_{n,k}$ as follows
\begin{align*}
\exc(g)&=(\alpha_1,\ldots,\alpha_{k})\\
\sub(g)&=(\beta_1,\ldots, \beta_{k}),
\end{align*}
where $\alpha_j$ (respectively $\beta_j$) counts the number of $j$-excedances (respectively $j$-subcedances) in $f$. It is easily seen that $|\exc(g)|+|\sub(g)|=n-1$.

Let $\alpu = \{u_1,u_2,\ldots\}$ and $\alpv = \{v_1,v_2,\ldots\}$ be two countably infinite sets of commuting indeterminates, which further commute with each other. This given, we associate a polynomial $G^{(1)}_{n,k}$ in the variables $\alpu$ and $\alpv$ with $\wtFunc_{n,k}$ as shown below.
\begin{align*}
G^{(1)}_{n,k}=\sum_{g\in\wtFunc_{n,k}} \alpu^{\exc(g)}\alpv^{\sub(g)}
\end{align*}
Our next two examples illustrate the various notions introduced above. The next example also serves to illustrate how we will think of elements of $\wtFunc_{n,k}$ as certain rook placements of `colored' rooks on an $(n-1)\times n$ rectangular board.
\begin{example}
\em{Let $n=6$ and $k=2$. The  maximal nonattacking rook placement below on the left corresponds to a unique element $f \in\Func_{6,2}$. The rook placement on the right corresponds to element $g:=\Phi(f)\in\wtFunc_{6,2}$.
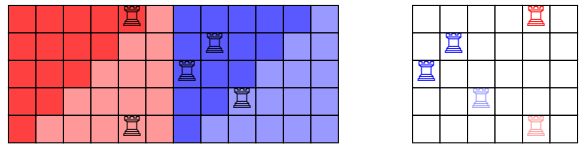
\begin{align*}
\ytableausetup{boxsize=1em}
\begin{ytableau}
*(red!75) & *(red!75) &*(red!75) &*(red!75) &*(red!75)\rook &*(red!40)  &*(blue!65) &*(blue!65) &*(blue!65) &*(blue!65) &*(blue!65) &*(blue!40) \\
*(red!75) & *(red!75) &*(red!75) &*(red!75) &*(red!40) &*(red!40)  &*(blue!65) &*(blue!65)\rook &*(blue!65) &*(blue!65)&*(blue!40) &*(blue!40) \\
*(red!75) & *(red!75) &*(red!75) &*(red!40) &*(red!40) &*(red!40)  &*(blue!65)\rook &*(blue!65) &*(blue!65) &*(blue!40) &*(blue!40) &*(blue!40)\\
*(red!75) & *(red!75) &*(red!40) &*(red!40) &*(red!40) &*(red!40) &*(blue!65) &*(blue!65) &*(blue!40)\rook &*(blue!40) &*(blue!40) &*(blue!40)  \\
*(red!75) & *(red!40) &*(red!40) &*(red!40) &*(red!40)\rook &*(red!40) &*(blue!65) &*(blue!40) &*(blue!40) &*(blue!40) &*(blue!40) &*(blue!40)  \\
\end{ytableau}
\hspace{10mm}
\begin{ytableau}
*(white) & *(white) &*(white) &*(white) &*(white)\textcolor{red}\rook &*(white) \\
*(white) &*(white)\textcolor{blue}{\rook} &*(white) &*(white) &*(white) &*(white) \\
*(white)\textcolor{blue}{\rook} & *(white) &*(white) &*(white) &*(white) &*(white)\\
*(white) &*(white) &*(white)\textcolor{blue!40}{\rook} &*(white) &*(white) &*(white) \\
*(white) &*(white) &*(white) &*(white) & *(white)\textcolor{red!40}{\rook} &*(white) \\
\end{ytableau}
\end{align*}
Note that the rook placement on the right has `colored' rooks, and is nonattacking in the sense that if two rooks occupy the same columns, then they must have distinct colors.
If we look in the $i$-th row for $1\leq i\leq 5$, then the color of the rook implicitly tells us about the second coordinate of $g(i)$ and the column in which it belongs gives us the first coordinate of $g(i)$. 

Note that $\exc(g)=(1,1)$ and $\sub(g)=(1,2)$, and therefore, the associated monomial with $g$ is $u_1u_2v_1v_2^2$.}
\end{example}

\begin{example}
\em{
We will quickly compute $G_{2,2}^{(1)}$ next. The four maximal nonattacking rook placements corresponding to $\Func_{2,2}$ are shown below.
\begin{align*}
\ytableausetup{boxsize=1em}
\begin{ytableau}
*(red!75)\rook & *(red!40) & *(blue!65) & *(blue!40)
\end{ytableau}\hspace{3mm}
\begin{ytableau}
*(red!75) & *(red!40)\rook & *(blue!65) & *(blue!40)
\end{ytableau}\hspace{3mm}
\begin{ytableau}
*(red!75) & *(red!40) & *(blue!65)\rook & *(blue!40)
\end{ytableau}\hspace{3mm}
\begin{ytableau}
*(red!75) & *(red!40) & *(blue!65)  & *(blue!40)\rook
\end{ytableau}\hspace{3mm}
\end{align*}
From these we obtain the four elements of $\wtFunc_{2,2}$ shown below.
\begin{align*}
\begin{ytableau}
*(white)\textcolor{red}{\rook} & *(white)
\end{ytableau}\hspace{10mm}
\begin{ytableau}
*(white) & *(white)\textcolor{red!40}{\rook}
\end{ytableau}\hspace{10mm}
\begin{ytableau}
*(white)\textcolor{blue}{\rook} & *(white)
\end{ytableau}\hspace{10mm}
\begin{ytableau}
*(white) & *(white)\textcolor{blue!40}{\rook}
\end{ytableau}
\end{align*}
Thus, we obtain that
\begin{align*}
G_{2,2}^{(1)}=v_1+u_1+v_2+u_2.
\end{align*}}
\end{example}

\subsection{Catalan, Shi, and Linial boards}
In this subsection, we consider certain special boards that are important for us. 
Let $t$ be a  nonnegative integer. We define the \emph{$t$-Catalan board}, denoted by $C_{t,n}$ to be the $(n-1)\times (n-2+t)$ board. We define the \emph{$t$-Shi board}, denoted by $S_{t,n}$ to be the Ferrers board with shape $(2n-3+t,\ldots,n-1+t)$. Finally, we define the \emph{$t$-Linial board}, denoted by $L_{t,n}$, to be the skew Ferrers board with shape $(2n-3+t,\ldots,n-1+t)/(n-1,\ldots,1)$. An easy computation establishes the following.
\begin{align*}
r_{n-1}(C_{t,n})&=\prod_{j=1}^{n-1}(t+n-1-j)=\sum_{j=0}^{n-1}r_{j}(C_{0,n})(t)_{n-1-j}=R(t,C_{0,n})\\
r_{n-1}(S_{t,n})&=(t+n-1)^{n-1}=\sum_{j=0}^{n-1}r_{j}(S_{0,n})(t)_{n-1-j}=R(t,S_{0,n}).
\end{align*}
From \cite[Theorem 2.4]{Headley} (also from \cite[Corollary 9.9.1]{Postnikov-Stanley} and \cite[Theorem 3.3]{Athanasiadis-Advances}), it follows that $\chi_{n-1}^{1,2}(q)=(q-n)^{n-1}$. Similarly, from \cite[Theorem 9.8]{Postnikov-Stanley}, it follows that $\chi_{n-1}^{1,1}=(q+1-n)(q+2-n)\cdots (q+n-1-n)$. We therefore infer the following equalities.
\begin{align}\label{eqn: catshi}
R(t,C_{0,n})&=(-1)^{n-1}\chi_{n-1}^{1,1}(1-t)\nonumber\\
R(t,S_{0,n})&=(-1)^{n-1}\chi_{n-1}^{1,2}(1-t) 
\end{align} 
At this point, the reader is invited to compare the equalities above with the statement of Theorem \ref{thm: intro 2}.
Furthermore, given Equation \eqref{eqn: catshi}, note that Theorem \ref{thm: GJW-v} allows us to write the characteristic polynomials of the extended Catalan and extended Shi arrangements as a sum that runs over the elements of partition lattice $\Pi_{n-1}$ using the identities below.
\begin{align}\label{eqn: catshi2}
R(t,C_{0,n})&=\sum_{\sigma\in \Pi_{n-1}}\mu(\hat{0},\sigma)\prod_{A\in \sigma}(t+n-2)\nonumber\\
R(t,S_{0,n})&=\sum_{\sigma\in \Pi_{n-1}}\mu(\hat{0},\sigma)\prod_{A\in \sigma}(t+n-2+\min(A))
\end{align}

The preceding discussion naturally raises the question of finding a similar interpretation for characteristic polynomials of extended Linial arrangements. As we shall soon see, the Linial boards will be crucial to this end. The route we follow is described next: Express $r_{n-1}(L_{t,n})$ as the factorial polynomial of $L_{0,n}$ and observe that this polynomial is of degree $n-1$ in $t$. Using the bijection we describe in the next section, we compute $r_{n-1}(L_{t,n})$ for infinitely many integer specializations of $t$ and relate it to the  the characteristic polynomial of extended Linial arrangement.

%\textcolor{red}{Insert example involving Linial board here or mention example earlier.}
\begin{example}
\em{
For $n=3$ and $t=2$, given below are the boards $C_{t,n}$, $S_{t,n}$ and $L_{t,n}$.
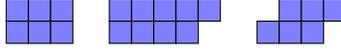
\begin{align*}
\ytableausetup{mathmode,boxsize=0.8em}
\begin{ytableau}
  *(blue!50) & *(blue!50) & *(blue!50)\\
 *(blue!50)& *(blue!50) & *(blue!50)\\
\end{ytableau}
\hspace{5mm}
\ytableausetup{mathmode,boxsize=0.8em}
\begin{ytableau}
 *(blue!50) & *(blue!50) & *(blue!50) & *(blue!50) & *(blue!50)\\
 *(blue!50)& *(blue!50) & *(blue!50) & *(blue!50) \\
\end{ytableau}
\hspace{5mm}
\ytableausetup{mathmode,boxsize=0.8em}
\begin{ytableau}
 \none & *(blue!50) & *(blue!50) & *(blue!50)\\
 *(blue!50)& *(blue!50) & *(blue!50)\\
\end{ytableau}
\end{align*}

}
\end{example}
\section{Trees}\label{sec:trees}
Given positive integers $n$ and $k$, we will denote the set of \emph{labeled rooted trees} on $n$ vertices such that any vertex has at most $k$ children by $\tree_{n,k}$. Furthermore, we will denote the set of \emph{labeled rooted plane $k$-ary trees} on $n$ vertices by $\ptree_{n,k}$. Given any vertex $v$ of a tree $T\in \ptree_{n,k}$, we will consider its children ordered from left to right, and by  the $i$-th child of $v$ we mean the $i$-th child from the left where $1\leq i\leq k$. We remark here that the trees in $\ptree_{n,k}$ are \emph{incomplete}, in that every nonleaf node need not $k$ children. Additionally, we will refer to the first child as the \emph{left child} and the $k$-th child as the \emph{right child}. We will want to distinguish between left and right children, and therefore, from this point onwards, we will work under the assumption that $k\geq 2$. For all our trees, we will implicitly assume that the edges of the tree are directed towards the root. 

A tree $T\in \tree_{n,k}$ can be endowed with some additional data so as to obtain a \emph{decorated tree}. Essentially, we assign to each edge connecting any parent node to its children in $T$ a distinct label from the set $[k]$. We denote the set of decorated trees on $n$ vertices such that every node has at most $k$ children by $\dtree_{n,k}$. Note that there is an obvious natural bijection between $\dtree_{n,k}$ and $\ptree_{n,k}$. Given $T\in\dtree_{n,k}$, for an edge labeled $i$ going from $p$ to $q$, we retain the edge but designate $q$ to be the $i$-th child of $p$ (and get rid of the labels of the edge). It is easy to see that in so doing, we obtain an element of $\ptree_{n,k}$, and that the aforementioned map is indeed a bijection.

Given $T\in \ptree_{n,k}$, define the \emph{spine} of $T$, denoted by $\spine(T)$ to be the sequence of numbers encountered along the unique path from the root to the vertex labeled $n$.  Suppose $\spine(T)=a_1\cdots a_m$ where $a_1$ is the label of the root and $a_m=n$. We will call an integer $1\leq i\leq m-1$ a \emph{record} of $T$ if $a_i>a_j$ for all $1\leq j< i$. Furthermore, we will denote the number of records in $T$ by $\rec(T)$.

\begin{remark}
Note that records of a labeled tree are a straightforward generalization of records or left-to-right maxima of a permutation.
\end{remark}

 Now given a tree $T\in \ptree_{n,k}$ and $1\leq i \leq k$, call an edge from $q$ to $p$ an \emph{$i$-descent} if the following two conditions are met.
\begin{enumerate}
\item $p>q$.
\item $q$ is the $i$-th child of $p$.
\end{enumerate}
We define an \emph{$i$-ascent} similarly by reversing the inequality in the first condition above. Using these notions we define certain special subsets of $\ptree_{n,k}$.
\begin{align*}
\iptree_{n,k}&=\{T\in \ptree_{n,k}\suchthat \text{$T$ has no $i$-descents for $1\leq i\leq k$}\}\\
\riptree_{n,k}&=\{T\in \ptree_{n,k}\suchthat \text{$T$ has no $k$-descents}\}\\
\ltree_{n,k}&=\{T\in \ptree_{n,k}\suchthat \text{$T$ has no $1$-ascents or $k$-descents}\}\\
\ltree_{n,k}^{b}&=\left\{T\in \ltree_{n,k}\suchthat \begin{array}{c}\text{$1$ has no right child and}\\\text{$n$ has no left child}\end{array}\right\}
\end{align*}
Clearly, elements of $\iptree_{n,k}$ are labeled rooted \emph{increasing} plane $k$-ary trees while elements of $\riptree_{n,k}$ are labeled rooted plane $k$-ary trees that always increase to the right. Both of these classes are extremely well studied classes, and it is the other two classes defined above that are important for our purposes. 

We will be interested in enumerating trees according to various constraints based on ascents and descents. To this end, it helps to introduce a multivariate generating function keeping track of these statistics. Firstly, we associate two weak compositions $\asc(T)=(\alpha_1,\ldots, \alpha_k)$ and $\dsc(T)=(\beta_1,\ldots, \beta_k)$ with $T\in \ptree_{n,k}$ as shown below.
\begin{align*}
\alpha_i &= |\{\text{$i$-descents in } T\}|\\
\beta_i &=|\{\text{$i$-ascents in } T\}|
\end{align*}
Clearly, we have that $|\alpha|+|\beta|=n-1$.
Next we associate a polynomial $G^{(2)}_{n,k}$ in the variables $\alpu$ and $\alpv$ with $\ptree_{n,k}$.
\begin{align*}
G^{(2)} _{n,k}=\sum_{T\in\ptree_{n,k}} \alpu^{\dsc(T)}\alpv^{\asc(T)}
\end{align*}

\begin{remark}
We could have easily associated analogues of the aforementioned statistics to decorated trees rather than labeled rooted plane $k$-ary trees.
\end{remark}

Our next goal is to establish that $G^{(1)}_{n,k}=G^{(2)}_{n,k}$. To this end, we will describe our all important bijection between $\wtFunc_{n,k}$ and $\ptree_{n,k}$. 
We would like to emphasize that our algorithm is a mild generalization of the bijection given in \cite{Remmel-Egecioglu}. The reader should read the construction described next in tandem with Example \ref{example: main bijection}.

Let $f\in \wtFunc_{n,k}$. We construct first the decorated digraph corresponding to $f$ as follows: For a positive integer $i$ satisfying $1\leq i\leq n-1$, let $f(i)=(a,b)$. Note that $1\leq a\leq n$ and $1\leq b\leq k$. Draw a directed edge from $i$ to $a$, and assign it the label $b$. Given the definition of $\wtFunc_{n,k}$, it is clear that there is no other edge directed towards $a$ whose label is $b$. Additionally, the digraph thus obtained possesses the following features.
\begin{itemize}
\item There is a unique connected component containing $n$, and it is a tree rooted at $n$.
\item The other connected components (if any) comprise directed cycles (with length $\geq 1$) of rooted trees. 
\end{itemize}
We will consider the connected components of the functional digraph thus obtained from left to right in increasing order of the largest entries in their directed cycles. Furthermore, the connected component that contains $n$ (which does not contain any directed cycles) will be the rightmost. Let these connected components be $T_1$ through $T_r$. For each $1\leq i\leq r-1$,  let $a_i$ denote the largest entry taking part in the directed cycle in $T_i$ and suppose further that the unique edge directed out of $a_i$ points towards $b_i$, and that the label of this edge is $c_i$. Furthermore, let $a_{r}=n$. Now for each $1\leq i\leq r-1$, convert each $T_i$ to a tree rooted at $a_i$  by removing the directed edge connecting $a_i$ to $b_i$ labeled $c_i$ (and remember this label). In this way, from the functional digraph, we obtain a sequence of rooted trees $T_1$ through $T_r$ with labeled edges. 

Next we will assemble this sequence of rooted trees into a single tree. For $1\leq i\leq r-1$, draw a directed edge from $a_{i+1}$ to $b_{i}$ labeled $c_i$. Note that this gives us a tree rooted at $a_1$. Furthermore, note that the procedure is reversible. That is, given a rooted tree with vertices labeled from $[n]$ and edges labeled from $[k]$, we can recover a unique functional digraph as follows: Consider the path from root to the vertex labeled $n$. Let this be comprising vertices labeled $j_1,\ldots,j_p$. Let $s$ and $t$ be the indices of two successive records in the sequence $j_1,\ldots, j_m$ with $s<t$. Remove the edge connecting $j_{t-1}$ to $j_t$. Suppose further that the label of this edge was $c$. Now insert an directed edge from $j_s$ to $j_{t-1}$ with label $c$. Doing this procedure for every pair of consecutive records gives us the corresponding functional digraph.

At this point, it is easy to obtain an element of $\ptree_{n,k}$. Indeed, if the edge from $i$ to $j$ has label $l$, then we set $i$ to be the $l$-th child of $j$. Doing this gives us a tree $T\in \ptree_{n,k}$.
Let us call the tree thus obtained $\Psi(f)$. It is easy to see given the preceding paragraph that $\Psi$ is a bijection between $\wtFunc_{n,k}$ and $\ptree_{n,k}$ (and therefore induces a bijection between $\Func_{n,k}$ and $\ptree_{n,k}$). 
\begin{example}\label{example: main bijection}
\em{
Let $n=21$ and $k=3$. Consider the following element of $\wtFunc_{21,3}$.

\begin{tabular}{ |p{0.5cm}|p{0.5cm}|p{0.5cm}|p{0.5cm}|p{0.5cm}|p{0.5cm}|p{0.5cm}|p{0.5cm}|p{0.5cm}|p{0.5cm}|p{0.5cm}|p{0.5cm}|p{0.5cm}|p{0.5cm}|p{0.5cm}|p{0.5cm}|p{0.5cm}|p{0.5cm}|p{0.5cm}|p{0.5cm}|}
 \hline
 \multicolumn{20}{|c|}{$f$} \\
 \hline
 1 & 2 & 3 & 4 & 5 & 6 & 7 & 8 & 9 & 10 & 11 & 12 & 13 & 14 & 15 & 16 & 17 & 18 & 19 & 20\\
 \hline
 3,1   & 5,2    & 4,3 &  5,3 & 3,3   & 21,2    & 7,1 &  12,2 & 1,3   & 4,2    & 4,1 &  20,2 & 19,3   & 19,1    & 6,2 &  1,2 & 16,1   & 6,3    & 7,2 & 12,3 \\
 \hline
\end{tabular}

Then we obtain the following decorated digraph corresponding to $f$.

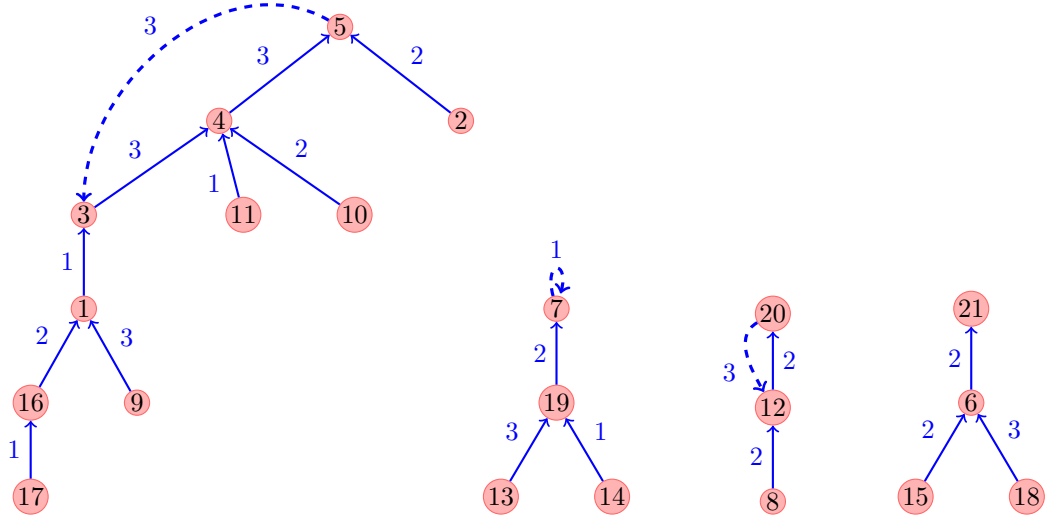
\begin{tikzpicture}[every tree node/.style={fill=red!30,draw=red!60,circle,inner sep=0.5pt,minimum size=2pt},
   level distance=1.25cm,sibling distance=1cm,
   edge from parent path={(\tikzparentnode) -- (\tikzchildnode)}]
\Tree
[.\node(5){$5$};
    \edge[<-,thick,blue] node[auto=right] {$3$};
    [.\node(4){$4$};
       \edge[<-,thick,blue] node[auto=right] {$3$};
       [.\node(3){$3$};
       \edge[<-,thick,blue] node[auto=right] {$1$};
       	[.\node(1){$1$};
       	\edge[<-,thick,blue] node[auto=right] {$2$};
       	[.\node(16){$16$};
       	\edge[<-,thick,blue] node[auto=right] {$1$};
       		[.\node(17){$17$};]]
       	\edge[<-,thick,blue] node[auto=left] {$3$};
       	[.\node(9){$9$};]
       	]
       ]
       \edge[<-,thick,blue] node[auto=right] {$1$};
       [.\node(11){$11$};]
       \edge[<-,thick,blue] node[auto=left] {$2$};
       [.\node(10){$10$};]
        ]
    \edge[<-,thick,blue] node[auto=left] {$2$};
    [.\node(2){$2$}; 
    ]
]
\path (5) edge[->, in=90,out=150,very thick, dashed, blue,auto=right] node {$3$} (3);
\end{tikzpicture}
\begin{tikzpicture}
[every tree node/.style={fill=red!30,draw=red!60,circle, inner sep=0.5pt, minimum size=2pt},
   level distance=1.25cm,sibling distance=1cm,
   edge from parent path={(\tikzparentnode) -- (\tikzchildnode)}]

\Tree
[
.\node(7){$7$};
\edge[<-,thick,blue] node[auto=right] {$2$};
	[
	.\node(19){$19$};
	\edge[<-,thick,blue] node[auto=right] {$3$};
		[.\node(13){$13$};]
	\edge[<-,thick,blue] node[auto=left] {$1$};	
		[.\node(14){$14$};]
	]
]
\path (7) edge [loop above, very thick, dashed, blue,auto=above] node {$1$} (7);
\end{tikzpicture}\hspace{10mm}
\begin{tikzpicture}
[every tree node/.style={ fill=red!30,draw=red!60,circle, inner sep=0.5pt, minimum size=2pt},
   level distance=1.25cm,sibling distance=1cm,
   edge from parent path={(\tikzparentnode) -- (\tikzchildnode)}]

\Tree
[
.\node(20){$20$};
\edge[<-,thick,blue] node[auto=left] {$2$};
	[
	.\node(12){$12$};
	\edge[<-,thick,blue] node[auto=right] {$2$};
		[.\node(8){$8$};]
	]
]
\path (20) edge[->, in=120,out=210,very thick, dashed, blue,auto=right] node {$3$} (12);
\end{tikzpicture}
\hspace{10mm} 
\begin{tikzpicture}
[every tree node/.style={fill=red!30, draw=red!60,circle,inner sep=0.5pt, minimum size=2pt},
   level distance=1.25cm,sibling distance=1cm,
   edge from parent path={(\tikzparentnode) -- (\tikzchildnode)}]

\Tree
[
.\node(21){$21$};
\edge[<-,thick,blue] node[auto=right] {$2$};
	[
	.\node(6){$6$};
	\edge[<-,thick,blue] node[auto=right] {$2$};
		[.\node(15){$15$};]
	\edge[<-,thick,blue] node[auto=left] {$3$};
		[.\node(18){$18$};]	
	]
]
\end{tikzpicture}

Applying the procedure outlined earlier, we obtain the following element of $\dtree_{21,3}$.
$$
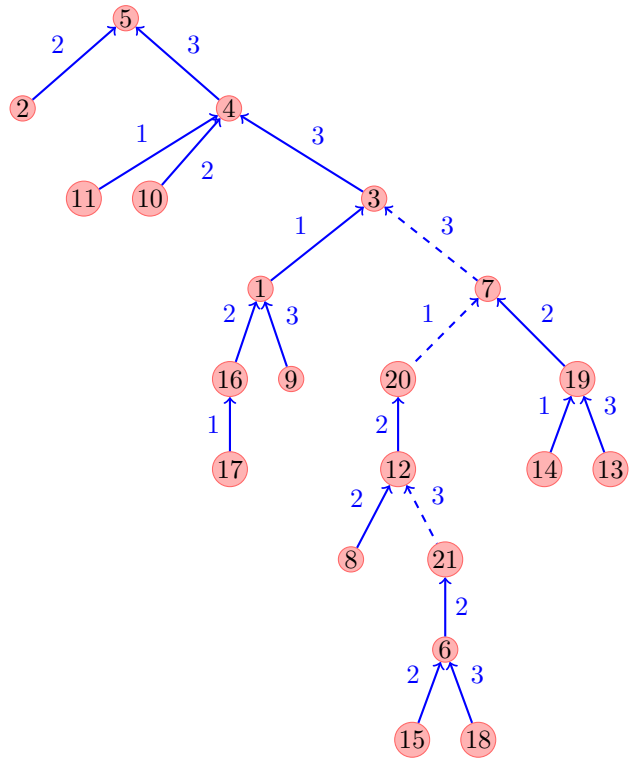
\begin{tikzpicture}[every tree node/.style={fill=red!30,draw=red!60,circle,inner sep=0.5pt,minimum size=4pt},
   level distance=1.2cm,sibling distance=0.4cm,
   edge from parent path={(\tikzparentnode) -- (\tikzchildnode)}]
\Tree
[.\node(5){$5$};
	\edge[<-,thick,blue] node[auto=right] {$2$};
	[.\node(2){$2$};
	]
	\edge[<-,thick,blue] node[auto=left] {$3$};
	[.\node(4){$4$};
		\edge[<-,thick,blue] node[auto=right]{$1$};
		[.\node(11){$11$};
		]	
		\edge[<-,thick,blue] node[auto=left]{$2$};
		[.\node(10){$10$};
		]	
		\edge[<-,thick,blue] node[auto=left]{$3$};
		[.\node(3){$3$};
			\edge[<-,thick,blue] node[auto=right]{$1$};
			[.\node(1){$1$};	
				\edge[<-,thick,blue] node[auto=right]{$2$};
				[.\node(16){$16$};
					\edge[<-,thick,blue] node[auto=right]{$1$};
					[.\node(17){$17$};]
				]	
				\edge[<-,thick,blue] node[auto=left]{$3$};
				[.\node(9){$9$};
			   ]				
			]		
			\edge[<-,thick,dashed,blue] node[auto=left]{$3$};
			[.\node(7){$7$};
				\edge[<-,thick,dashed,blue] node[auto=right]{$1$};
				[.\node(20){$20$};	
					\edge[<-,thick,blue] node[auto=right]{$2$};
					[.\node(12){$12$};	
						\edge[<-,thick,blue] node[auto=right]{$2$};
						[.\node(8){$8$};]					
					\edge[<-,thick,dashed,blue] node[auto=left]{$3$};
					[.\node(21){$21$};
						\edge[<-,thick,blue] node[auto=left]{$2$};
						[.\node(6){$6$};	
							\edge[<-,thick,blue] node[auto=right]{$2$};
							[.\node(15){$15$};
							]	
							\edge[<-,thick,blue] node[auto=left]{$3$};
							[.\node(18){$18$};
							]						
						]						
					]				
				]
				]	
				\edge[<-,thick,blue] node[auto=left]{$2$};
				[.\node(19){$19$};
					\edge[<-,thick,blue] node[auto=right]{$1$};
					[.\node(14){$14$};
					]		
					\edge[<-,thick,blue] node[auto=left]{$3$};
					[.\node(13){$13$};
					]				
				]			
			]		
		]	
	]
]
\end{tikzpicture}
$$
One can now easily obtain $\Psi(f)\in\ptree_{21,3}$ by orienting the edges according to the labels they carry.}
\end{example}
This brings us to our next theorem detailing the properties of $\Psi$.

\begin{theorem}
The map $\Psi$ is a bijection from $\wtFunc_{n,k}$ and $\ptree_{n,k}$. The number of connected components in the functional digraph corresponding to $f\in \wtFunc_{n,k}$ is $\rec(\Psi(f))$. Furthermore, we have that $\exc(f)=\dsc(\Psi(f))$ and $\sub(f)=\asc(\Psi(f))$.
\end{theorem}
\begin{proof}

We have already established earlier that $\Psi$ is a bijection, and it is implicit in the argument above that the number of connected components in the functional digraph corresponding to $f\in \wtFunc_{n,k}$ is $\rec(\Psi(f))$.
Now let $\exc(f)=(\alpha_1,\ldots,\alpha_k)$ and $\sub(f)=(\beta_1,\ldots,\beta_k)$. Additionally, let $\asc(\Psi(f))=(\gamma_1,\ldots,\gamma_k)$ and $\dsc(\Psi(f))=(\delta_1,\ldots,\delta_k)$.

For the length of this proof, let $T:=\Psi(f)$. Note that we have
\begin{align*}
&\delta_i=\text{\# edges $b\to a$ in $T$ such that $a>b$ and $b$ is the $i$-th child of $a$},\\
&\alpha_i=\text{\# edges $b\to a$ in functional digraph associated to $f$ such that $a>b$ and edge is labeled $i$}.
\end{align*}
In recovering $f$ from $T$, it suffices to recover the functional digraph of $f$. To do this, we identify the spine of $T$, and redirect certain edges leaving the rest unaltered. Observe that the redirected edges do not correspond to descents in $T$, and additionally, upon redirection they point from a larger number to a smaller number (or to itself), in the decorated functional digraph thus obtained. Thus, it follows that 
$\delta_i=\alpha_i$, and therefore that $\dsc(T)=\exc(f)$.

Now we will establish that $\asc(T)=\sub(f)$ with a similar argument. Note that in going back to the decorated functional digraph from $T$, we redirect certain ascent edges $b\to a$ where $b$ is the $i$-th child of $a$ and $b>a$. But despite the redirection, they continue to point from a larger number to a smaller number and the edge label is $i$, inherited from before. This immediately implies that $\asc(T)=\sub(f)$.
\end{proof}

As a corollary of the above theorem, we obtain the following.
\begin{corollary}
We have that $
G_{n,k}^{(1)}=G_{n,k}^{(2)}$.
\end{corollary}
In view of the above corollary, we define $G_{n,k}:=G_{n,k}^{(1)}=G_{n,k}^{(2)}$ to be the \emph{Gessel polynomial} indexed by positive integers $n$ and $k$.

\begin{remark}\label{remark: work of Gessel}
The polynomials $G_{n,2}$ have been studied by Gessel \cite{Gessel-unpublished}, Kalikow \cite{Kalikow} and Drake \cite{Drake-thesis}. More precisely, the following generating function was considered by the aforementioned authors.
\begin{align*}
B(u_1,u_2,v_1,v_2;x)=\sum_{n\geq 1}G_{n,2}(u_1,u_2,v_1,v_2)\frac{x^n}{n!}
\end{align*}
It develops that $B:=B(u_1,u_2,v_1,v_2;x)$ satisfies the following functional equation.
\begin{align*}
\frac{(1+v_1B)(1+u_2B)}{(1+v_2B)(1+u_1B)}=e^{[(v_1u_2-v_2u_1)B+v_1-v_2-u_1+u_2]x}
\end{align*}
It is worthwhile to note the non-obvious symmetry in $G_{n,2}$ from the equation above, namely that  $$G_{n,2}(u_1,u_2,v_1,v_2)=G_{n,2}(u_1,v_1,u_2,v_2).$$ A recursive proof has been obtained by Kalikow \cite{Kalikow}. Note further that, from classical results of rooted labeled binary trees, we know that $G_{n,2}(1,1,1,1)=n! C_n$ where $C_n=\frac{1}{n+1}\binom{2n}{n}$ denotes the $n$-th Catalan number. Additionally, $G_{n,2}(1,1,1,0)=(n+1)^{n-1}$, $G_{n,2}(1,0,1,0)=n!$ and $G_{n,2}(u_1,u_2,0,0)=\sum_{\sigma\in \mathfrak{S}_n}u_1^{\dsc(\sigma)}u_2^{\asc(\sigma)}$ (the $n$-th homogenized Eulerian polynomial). The careful reader might recall that the number of regions in the Catalan arrangement is $n! C_n$, the number of regions in the Shi arrangement is $(n+1)^{n-1}$, and the number of regions in the braid arrangement is $n!$.
\end{remark}

\subsection{Drake's results}
In this subsection we describe a beautiful result of Drake \cite{Drake-thesis} which allows us to compute the generating function for the ascent-descent statistics on $\ptree_{n,k}$. First, consider the following generating function.
\begin{align*}
P_k(\alpu,\alpv;x)=\sum_{n\geq 1}\sum_{T\in \ptree_{n,k}}\alpu^{\dsc(T)}\alpv^{\asc(T)}\frac{x^n}{n!}=\sum_{n\geq 1}G_{n,k}\frac{x^n}{n!}
\end{align*}
The following theorem of Drake \cite[Theorem 1.8.4]{Drake-thesis} expresses $P_k:=P_k(\alpu,\alpv;x)$ as a compositional inverse of another function.
% the exact place is Theorem 1.8.4 in Drake's thesis.
\begin{theorem}[\cite {Drake-thesis}]\label{thm: drake}
Consider an auxiliary function $Z(a,b,c,d)=(ad-bc)x+(a-c)-(b-d)$. Then we have that
\begin{align*}
P_k=\left(\sum_{i=1}^{k}(v_i-u_i)^{k-2}\log{\left(\frac{1+v_ix}{1+u_ix}\right)}\prod_{j\in [k], i\neq j}\frac{1}{Z(v_i,u_i,v_j,u_j)}\right)^{\langle -1 \rangle}
\end{align*}
\end{theorem}
The reader is invited to compare the above functional equation to the one obtained for $P_2$ in Remark \ref{remark: work of Gessel}. 
We will now use the above result to compute the cardinality of $\ltree_{n,k}$ where $k\geq 2$. Thus, we are counting the number of trees in $\ptree_{n,k}$ that constrained to have no $1$-ascents and no $k$-descents. Let $l_{n,k}$ denote this number, that is, $l_{n,k}=|\ltree_{n,k}|$. We can obtain the exponential generating function for these number by evaluating $P$ at the values $u_k=0,v_1=0$ and all the other variables set to $1$. 
Consider
\begin{align*}
\widetilde{P}_k=\sum_{n\geq 1}l_{n,k}\frac{x^n}{n!}.
\end{align*}
Then from Theorem \ref{thm: drake}, we get the following functional equation for $\widetilde{P}_k$ (the details are similar to the many examples discussed in \cite[Chapter 1.9]{Drake-thesis}).
\begin{align*}
\widetilde{P}_k=\left(\frac{2\log(1+x)}{(1+x)^{k-2}(2+x)}\right)^{\langle -1\rangle}
\end{align*}
Substituting $f=1+\widetilde{P}_k$, we get that $f$ satisfies the following relation.
\begin{align*}
f^2=e^{x(f^{k-2}+f^{k-1})}
\end{align*}
We see that the functional equation above is precisely the one in the statement of \cite[Theorem 9.3]{Postnikov-Stanley} with $b=k$ and $a=k-2$. Thus we see that the coefficient of $\frac{x^n}{n!}$ in $f$ for $n\geq 1$, that is, $l_{n,k}$,
is given by the following expression.
\begin{align}\label{eqn: local k-ary search trees}
l_{n,k}=\frac{1}{2^n}\sum_{j=0}^{n}\binom{n}{j}(1+(k-2)n+j)^{n-1}
\end{align}
\begin{remark}
In the case $k=2$, we are in fact counting local binary search trees, and the formula obtained above coincides with the formula obtained by Postnikov \cite{Postnikov-JCTA} and Athanasiadis \cite[Theorem 4.2]{Athanasiadis-Advances}.
\end{remark}

Now we return to our question of counting rook placements on skew Ferrers boards. Consider the sequence of Linial boards $L_{t,n}$ for $t=(a-1)n+2$ where $a$ ranges over all positive integers. For a fixed positive integer $a$, maximal nonattacking rook placements on $L_{(a-1)n+2,n}$ are clearly in bijection with the subset of $\Func_{n,a+1}$ that comprises injective functions $f:[n-1] \to [(a+1)n]$ satisfying the criterion $i+1\leq f(i)\leq an+i$ for $1\leq i\leq n-1$. But by our bijection $\Psi$ from before, the number of such rook placements is the number of trees $T\in \ptree_{n,a+1}$ such that there are no $1$-ascents and no $a+1$-descents. The condition $i+1\leq f(i)$ implies that $T$ can have no $1$-ascents while the condition $f(i)\leq an+i$ implies that $T$ can have no $a+1$-descents. This motivates our next theorem that gives a closed form expression for the factorial polynomial of the Linial board $L_{0,n}$.

\begin{theorem}\label{thm: factorial rook polynomial Linial board}
We have that
\begin{align*}
R(t,L_{0,n})=\frac{1}{2^n}\sum_{j=0}^{n}\binom{n}{j}(t-1+j)^{n-1},
\end{align*}
\begin{proof}
Note that since $L_{0,n}$ has $n-1$ rows, the factorial polynomial $R(t,L_{0,n})$ is a polynomial of degree $n-1$ by its definition. Furthermore, its definition also implies that the evaluations of $R(t,L_{0,n})$ at $t=m$ for $m$ a nonnegative integer agree with $r_{n-1}(L_{m,n})$. By comparison with Equation \eqref{eqn: local k-ary search trees}, we know that 
\begin{align*}
r_{n-1}(L_{m,n})=\frac{1}{2^n}\sum_{j=0}^{n}\binom{n}{j}(1+(m-2)+j)^{n-1}
\end{align*}
when $m$ is of the form $(a-1)n+2$ for a positive integer $a\geq 1$. Thus the polynomial of degree $n-1$ below:
\begin{align*}
\frac{1}{2^n}\sum_{j=0}^{n}\binom{n}{j}(1+(t-2)+j)^{n-1},
\end{align*}
agrees with $R(t,L_{0,n})$ for infinitely many integer values $t$. Therefore the two must be equal.
\end{proof}
\end{theorem}

Of course, given the similarity between the expression obtained in the theorem above and \cite[Equation 9.11]{Postnikov-Stanley}\cite[Theorem 4.2]{Athanasiadis-Advances}, the next corollary is no surprise.
\begin{corollary}\label{cor: relation factorial and characteristic}
Given a positive integer $a$, the following hold.
\begin{align*}
(-1)^{n-1}\chi_{n-1}^{a-1,a+1}(q)&=R(1+(a-1)n-q, L_{0,n})\\
(-1)^{n-1}\chi_{n-1}^{a-1,a+1}(q)&=\sum_{\sigma\in \Pi_{n-1}}\mu(\hat{0},\sigma)\prod_{A\in\sigma}(an-q-1+\min(A)-\max(A))
\end{align*}
\end{corollary}
\begin{proof}
The first equality follows from comparing the expression for $\chi_{n-1}^{a-1,a+1}$ in \cite[Equation 9.11]{Postnikov-Stanley} to the expression for $R(t,L_{0,n})$ in the  statement of Theorem \ref{thm: factorial rook polynomial Linial board}. 

For the second equality, we apply Theorem \ref{thm: GJW-v} to obtain an expression for the factorial polynomial of $L_{0,n}$, and use the fact that the statistic $v_{L_{0,n}}(A)=n-2-(\max(A)-\min(A))$ for $A\subseteq [n-1]$.
\end{proof}
Note that in view of the above corollary and Equation \eqref{eqn: catshi2}, we have expressed characteristic polynomials of the extended Catalan, extended Shi and extended Linial arrangements as sums over the lattice of set partitions. It is interesting to note that similar results were achieved in \cite{Forge-and-gang} from a very different perspective, and the equivalence with our results is not obvious. Furthermore, the case $a=1$ in the above corollary implies Theorem \ref{thm: intro 2} from the introduction. Finally, the case $a=1$ and $q=-1$ gives us the following alternate formula for the number of local binary search trees on $n$ nodes.
\begin{align}\label{eqn: alternate lbs}
|\ltree_{n,2}|=\sum_{\sigma\in \Pi_{n-1}}\mu(\hat{0},\sigma)\prod_{A\in\sigma}(n+\min(A)-\max(A))
\end{align}
Admittedly, this is not as nice as the more well-known closed form $\frac{1}{2^n}\sum_{j=0}^{n}\binom{n}{j}(1+j)^{n-1}$, but it would be interesting to relate the said closed form to the expression in Equation \eqref{eqn: alternate lbs}.

Now we are ready to state our main enumerative result, a special case of which was present in the introduction in the form of Theorem \ref{thm: intro 1}.
\begin{corollary}\label{cor: enumerative stuff}
For a fixed positive integer $a$,  let $\mathcal{A}:=\hyplin_{n-1,a}=\mathcal{A}^{a-1,a+1}_{n-1}$. Then the following hold.
\begin{enumerate}
\item $r(\mathcal{A})=R((a-1)n+2,L_{0,n})$ and $b(\mathcal{A})=R((a-1)n,L_{0,n})$.
\item $r(\mathcal{A})=r_{n-1}(L_{(a-1)n+2,n})$ and $b(\mathcal{A})=r_{n-1}(L_{(a-1)n,n})$.
\item $r(\mathcal{A})=|\ltree_{n,a+1}|$ and $b(\mathcal{A})=|\ltree_{n,a+1}^{b}|$.
\item $r(\mathcal{A})$ equals the number of sequences $(x_1,\ldots,x_{n-1})$ such that $1\leq x_i\leq an$ for $1\leq i\leq n-1$  and  $x_i+i\neq x_j+j$ for $1\leq i\neq j\leq n-1$; and $b(\mathcal{A})$ equals the number of sequences $(x_1,\ldots,x_{n-1})$ such that $1\leq x_i\leq an-2$ for $1\leq i\leq n-1$  and  $x_i+i\neq x_j+j$ for $1\leq i\neq j\leq n-1$.
\end{enumerate}
\end{corollary}
\begin{proof}
The first claim follows from plugging $q=\pm 1$ in Corollary \ref{cor: relation factorial and characteristic} and recalling Zaslavsky's result, that is, Theorem \ref{thm: Zaslav}. The second claim follows from the first by utilizing the relation between factorial polynomial and maximal nonattacking rook placements.

The discussion preceding Theorem \ref{thm: factorial rook polynomial Linial board} establishes that $r_{n-1}(L_{(a-1)n+2,n})=|\ltree_{n,a+1}|$ from which the first part of the third claim follows. To see how the second part follows, note that maximal nonattacking rook placements on $L_{(a-1)n,n}$ can be thought of maximal nonattacking rook placements on $L_{(a-1)n+2,n}$ where we forbid the rooks to belong to columns $n$ and $an+1$ (one can forbid any two columns of length $n-1$ in fact, but our choice leads to a nicer combinatorial interpretation for the number of bounded regions in extended Linial arrangements). Under the map $\Psi$, maximal nonattacking rook placements on $L_{(a-1)n+2,n}$ map to elements of $\ltree_{n,a+1}$. 
This given, the condition that we are forbidding rooks to belong to columns $n$ and $an+1$ translates to trees in $\ltree_{n,a+1}$ which satisfy the criterion that the node labeled $1$ does not have a right child and the node labeled $n$ does not have a left child. But these are precisely trees that belong to $\ltree_{n,a+1}^{b}$. This establishes the third claim.

The final claim follows easily once one left-justifies our Linial boards. The columns occupied by the rooks in the left-justified board give us our sequence $(x_1,\ldots, x_{n-1})$ (reading the rows from bottom to top) and the condition $x_i+i\neq x_j+j$ is tantamount to the nonattacking condition.\end{proof}
\section{Chromatic polynomials and Linial graphs}\label{sec: chromatic poly}
We will now associate a bipartite graph $G_{t,n}$ with the $t$-Linial board $L_{t,n}$. For fixed values of $t$ and $n$, consider two sets of vertices $A_1=\{v_1,\ldots,v_{2n-4+t}\}$ and $A_2=\{v_{2n-3+t},\ldots,v_{3n-5+t}\}$. Note that $|A_1|=2n-4+t$ and $|A_2|=n-1$. Now using these two sets we will construct our bipartite graph $G_{t,n}$ as follows. For $1\leq i\leq n-1$, we connect vertex $v_{2n-4+t+i}\in A_2$ to the vertices $v_i,v_{i+1},\ldots,v_{n-3+t+i}$ (all belonging to $A_1$). 

Note that all we are doing is labeling the rows of $L_{t,n}$ by elements of $A_2$ bottom to top and the columns of $L_{t,n}$ by elements of $A_1$ left to right and finally, drawing an edge for every square belonging to $L_{t,n}$. From this viewpoint, it is immediate that a nonattacking rook placement on $L_{t,n}$ corresponds to a matching on $G_{t,n}$. Recall that a \emph{matching} in a graph is a set of pairwise non-adjacent edges. Furthermore, a matching is said to be \emph{maximum} if it contains the largest possible number of edges.

The obvious relation between nonattacking rook placements and matchings gives us our next theorem.
\begin{theorem}\label{thm: rooks and matchings}
For a fixed positive integer $a$, let $\mathcal{A}:=\hyplin_{n-1,a}$. Then we have that 
\begin{align*}
r(\mathcal{A})&=\text{ number of maximum matchings in $G_{(a-1)n+2,n}$}\\
b(\mathcal{A})&= \text{ number of maximum matchings in $G_{(a-1)n,n}$}.
\end{align*}
\end{theorem}
While the above theorem is not surprising given Corollary \ref{cor: enumerative stuff}, it is interesting to note that the complement of the graph $G_{t,n}$, denoted by $\overline{G_{t,n}}$, has some combinatorial value in the following sense: $$\text{The chromatic polynomial of $\overline{G_{t,n}}$ coincides with the $(3n-5+t)$-factorial polynomial of $L_{t,n}$.}$$ 
We explain this briefly next. Observe that $3n-5+t$ equals the sum of the number of rows and the number of columns in $L_{t,n}$, and is also equal to the total number of vertices in $\overline{G_{t,n}}$. Furthermore, it develops that, upto an interchange of rows and columns, the graph $G_{t,n}$ is obtained from a proper $(3n-5+t)$-board.  The result follows from invoking \cite[Theorem 2]{GJW-iii}. We refer the reader to \cite{GJW-iii} for the definition of proper boards and more details on the relation between factorial polynomials of proper boards and chromatic polynomials of associated graphs.

\begin{theorem}\label{thm: relation chromatic and characteristic polynomial}
Denote the chromatic polynomial of $\overline{G_{t,n}}$ by $c(x,\overline{G_{t,n}})$. Let $a$ be a fixed positive integer. Then
\begin{align*}
\frac{1}{(x)_{2n-4+t}}c(x,\overline{G_{t,n}})=(-1)^{n-1}\,\chi_{n-1}^{a-1,a+1}((a+1)n-3-x)
\end{align*}
\end{theorem}
\begin{proof}
By \cite[Theorem 2]{GJW-iii} we know that 
\begin{align*}
c(x,\overline{G_{t,n}})=R_{3n-5+t}(x,L_{t,n})
\end{align*}
This given we have the following equalities.
\begin{align*}
c(x,\overline{G_{t,n}})&=
(x)_{2n-4+t}\left(\sum_{k=0}^{n-1}r_k(L_{t,n})(x-2n+4-t)_{n-1-k}\right)\\
&=
(x)_{2n-4+t}R(x-2n+4-t,L_{t,n})
\\&=
(x)_{2n-4+t}R(x-2n+4,L_{0,n})
\end{align*}
In the last step above, we have made use of the fact that $R(x,L_{t,n})=R(x+t,L_{0,n})$. Now from Corollary \ref{cor: relation factorial and characteristic}, we know that $R(x-2n+4,L_{0,n})=(-1)^{n-1}\chi_{n-1}^{a-1,a+1}((a+1)n-3-x)$, thereby establishing the claim.
\end{proof}
Thus, the content of Theorem \ref{thm: relation chromatic and characteristic polynomial} is that upto a certain normalizing factor, the chromatic polynomial of the graphs $\overline{G_{t,n}}$ coincides with the characteristic polynomial of extended Linial arrangements. In fact, the normalizing factor $(x)_{2n-4+t}$ can be thought of the chromatic polynomial of the complete graph on $2n-4+t$ vertices, that is, the complete graph on the vertices from the set $A_1$. Note that this complete graph is a subgraph of $\overline{G_{t,n}}$.

\begin{example}
\em{
Consider the case where $n=3$, $t=1$ and $a=1$. Drawn below is the Linial board $L_{1,3}$.
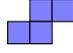
\begin{align*}
\ytableausetup{mathmode,boxsize=0.8em}
\begin{ytableau}
 \none & *(blue!50) & *(blue!50)\\
 *(blue!50)& *(blue!50)\\
\end{ytableau}
\end{align*}
The orange edges in the graph below on 5 vertices compose $G_{1,3}$, while the blue edges on the same set of vertices correspond to $\overline{G_{1,3}}$.
$$
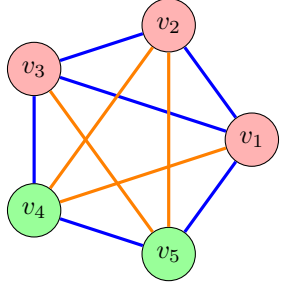
\begin{tikzpicture}
\node[draw, circle,fill=red!30] (v1) at ({360/5 * (1 - 1)}:1.6cm) {$v_1$};
\node[draw, circle,fill=red!30] (v2) at ({360/5 * (2 - 1)}:1.6cm) {$v_2$};
\node[draw, circle,fill=red!30] (v3) at ({360/5 * (3- 1)}:1.6cm) {$v_3$};
\node[draw, circle,fill=green!40] (v4) at ({360/5 * (4 - 1)}:1.6cm) {$v_4$};
\node[draw, circle,fill=green!40] (v5) at ({360/5 * (5 - 1)}:1.6cm) {$v_5$};
\draw[blue,very thick] (v1)--(v2);
\draw[blue,very thick] (v1)--(v3);
\draw[blue,very thick] (v2)--(v3);
\draw[blue,very thick] (v4)--(v5);
\draw[blue,very thick] (v1)--(v5);
\draw[blue,very thick] (v4)--(v3);
\draw[orange, very thick] (v1)--(v4);
\draw[orange, very thick] (v2)--(v4);
\draw[orange, very thick] (v2)--(v5);
\draw[orange, very thick] (v3)--(v5);
\end{tikzpicture}
$$
Now, a computation reveals that $c(x,\overline{G_{1,3}})=x(x-1)(x-2)(x^2-3x+3)$. Additionally, we have that $\chi_{2}^{0,2}(x)=x^2-3x+3=\chi_{2}^{0,2}(3-x)$. We can now immediately verify that $\frac{1}{(x)_3}c(x,\overline{G_{1,3}})=\chi_{2}^{0,2}(3-x)$.}
\end{example}

\section{Final remarks}\label{sec:final}
We close with some remarks about future avenues and related work.
\begin{enumerate}
\item It is natural to ask if we can extend our results to the entire class of truncated affine arrangements; or even to subarrangements thereof.
The Shi($G$) arrangements \cite{Athanasiadis-Advances, Athanasiadis-Linusson} and the $G$-Shi arrangements \cite{Hopkins-Perkinson} have garnered plenty of attention in the recent past, and remain an active area of research. A rook-theoretic perspective on them might be of interest and lend new insight. Furthermore, it would be interesting to obtain $q$-analogs of our formulae and extend the results of Haglund-Remmel \cite{Haglund-Remmel} to our skew boards. Finally, it might be worthwhile to explore the relation to work of Ardila on Tutte polynomials associated with hyperplane arrangements \cite{Ardila}. In work in progress \cite{Griffin-Tewari}, we aim to answer some of the aforementioned  questions.

\item We will focus on the case $k=2$ of our bijection. We know that the number of maximal nonattacking rook placements on the skew Ferrers board $(2n-1,\ldots,n+1)/(n-1,\ldots,1)$ equals the number of local binary search trees  on $[n]$. Clearly, the set of such rook placements is equinumerous with the set of bijections $\{f: \mathbb{Z}\mapsto \mathbb{Z}\}$ satisfying the following conditions:
\begin{enumerate}
\item $f(i+n-1)=f(i)+(n-1),$
\item $i\leq f(i)\leq i+n-1$.
\end{enumerate}
It is worth remarking that this is a subset of conditions used to define \emph{bounded affine permutations} \cite{Knutson-Lam-Speyer} and we wonder if there is a stronger connection. Also, nonattacking rook placements on skew Ferrers boards have also been studied in \cite{Postnikov, Sjostrand} and it might be interesting to put our results in the context of the aforementioned articles.
\item If we consider the rook placements corresponding to increasing $k$-ary labeled rooted trees on $n$ nodes, we see that we are essentially counting maximal nonattacking rook placements on the Ferrers board of shape $(k(n-1),k(n-2),\ldots,k)$. Then it is easy to obtain the equality
\begin{align*}
|\iptree_{n,k}|=\prod_{i=1}^{n-1}(1+i(k-1)).
\end{align*}
Josuat-Verg\`es \cite{Josuat-Verges} studied the case $k=2$ using stammering tableaux and the language of growth diagrams. It would be interesting to explore the analogue of stammering tableaux for $k>2$.
\end{enumerate}
Finally, we would be remiss to not mention other work inspired by Gessel's original question, namely \cite{Corteel-Forge-Ventos,Forge}, and more recently \cite{Bernardi} which answers Gessel's questions from a different perspective than ours. Connections with the aforementioned work, especially \cite{Bernardi}, and ours remain to be explored.

\section*{Acknowledgements}
I owe a lot to Ira Gessel for his generosity in sharing his ideas, wisdom and notes, as well as his prompt and detailed responses to all my queries. I would also like to thank Joel Lewis and Alejandro Morales for an invaluable push in the right direction and helping me view my results in the right context; further thanks go to the former for helpful suggestions on improving exposition. I  have benefitted enormously from the many suggestions of Christos Athanasiadis, Sara Billey and Richard Stanley regarding avenues to pursue and pointers to relevant literature.
Finally, I would like to thank Sean Griffin, Brendon Rhoades, Jos\'e Samper, Andy Wilson, and Alex Woo for many enlightening discussions and numerous helpful comments.

\end{document}